\documentclass[a4paper,11pt,reqno,noindent]{amsart}
\usepackage[centertags]{amsmath}
\usepackage{amsfonts,amssymb,amsthm,dsfont,cases,amscd,esint,enumerate}
\usepackage[T1]{fontenc}
\usepackage[english]{babel}
\usepackage[applemac]{inputenc}
\usepackage{newlfont}
\usepackage{color}
\usepackage[body={15cm,21.5cm},centering]{geometry} 
\usepackage{fancyhdr}
\usepackage{enumerate}

\theoremstyle{plain}
\newtheorem{theor10}{Theorem}

\newtheorem{prop10}{Proposition}

\newtheorem{cor10}{Corollary}

\newtheorem{theor0}{Theorem}[section]
\newenvironment{theor}
  {\pushQED{\qed}\begin{theor0}}
  {\popQED\end{theor0}}
\newtheorem{lem0}[theor0]{Lemma}
\newenvironment{lem}
  {\pushQED{\qed}\begin{lem0}}
  {\popQED\end{lem0}}
\newtheorem{prop0}[theor0]{Proposition}
\newenvironment{prop}
  {\pushQED{\qed}\begin{prop0}}
  {\popQED\end{prop0}}
\newtheorem{cor0}[theor0]{Corollary}
\newenvironment{cor}
  {\pushQED{\qed}\begin{cor0}}
  {\popQED\end{cor0}}
\newtheorem{propr0}[theor0]{Property}

\newtheorem{hyp0}[theor0]{Hypothesis}

\newtheorem{result0}[theor0]{Result}

\newtheorem{conj0}[theor0]{Conjecture}
\newenvironment{conj}
  {\pushQED{\qed}\begin{conj0}}
  {\popQED\end{conj0}}
\newtheorem{heur0}[theor0]{Heuristics}

\theoremstyle{definition}
\newtheorem{defin0}[theor0]{Definition}

\newtheorem{rems0}[theor0]{Remarks}

\newtheorem{ex0}[theor0]{Example}

\newtheorem{exs0}[theor0]{Examples}

\newtheorem{rem0}[theor0]{Remark}
\newenvironment{rem}
  {\pushQED{\qed}\begin{rem0}}
  {\popQED\end{rem0}}
\newtheorem{qu0}[theor0]{Question}

\newtheorem{qus0}[theor0]{Questions}

  \newtheorem{as0}[theor0]{Assumption}

\mathchardef\emptyset="001F
\numberwithin{equation}{section}

\newcommand{\N}{\mathbb N}

\newcommand{\e}{\varepsilon}

\newcommand{\Hc}{\mathcal{H}}

\newcommand{\R}{\mathbb R}
\newcommand{\Z}{\mathbb Z}

\newcommand{\F}{\mathcal F}

\newcommand{\loc}{{\operatorname{loc}}}
\newcommand{\Id}{\operatorname{Id}}

\newcommand{\per}{{\operatorname{per}}}

\newcommand{\ee}{e}
\newcommand{\Aa}{\boldsymbol a}
\newcommand{\bb}{{\boldsymbol b}}

\newcommand{\Ld}{\operatorname{L}}

\newcommand{\step}[1]{\noindent \textit{Step} #1.}

\newcommand{\Pm}{\mathbb{P}}

\newcommand{\expec}[1]{\mathbb{E}\left[ #1 \right]}
\newcommand{\expecm}[1]{\mathbb{E}\big[ #1 \big]}

\usepackage{hyperref}

\title[A remark on a result by Bourgain in homogenization]{A remark on a surprising result by Bourgain\\in homogenization}
\author[M. Duerinckx]{Mitia Duerinckx}
\author[A. Gloria]{Antoine Gloria}
\author[M. Lemm]{Marius Lemm}
\address[Mitia Duerinckx]{École Normale Supérieure de Lyon, UMR 5669, Unité de Mathématiques Pures et Appliquées, Lyon, France \& Universit\'e Libre de Bruxelles, Département de Mathématique, Brussels, Belgium}
\email{mduerinc@ulb.ac.be}
\address[Antoine Gloria]{Sorbonne Universit\'e, UMR 7598, Laboratoire Jacques-Louis Lions, Paris, France \& Universit\'e Libre de Bruxelles, Département de Mathématique, Brussels, Belgium}
\email{gloria@ljll.math.upmc.fr}
\address[Marius Lemm]{Institute for Advanced Study, School of Mathematics, Princeton, USA \& Harvard University, Department of Mathematics, Cambridge, USA}
\email{mlemm@math.harvard.edu}

\begin{document}
\selectlanguage{english}

\begin{abstract}
In a recent work, Bourgain gave a fine description of the expectation of solutions of discrete linear elliptic equations on $\Z^d$ with random 
coefficients in a perturbative regime using tools from harmonic analysis. 
This result is surprising for it goes beyond the expected accuracy suggested by recent results in quantitative stochastic homogenization.
In this short article we reformulate Bourgain's result in a form that highlights its interest to the state-of-the-art in homogenization (and especially the
theory of fluctuations), and we state several related conjectures.
\end{abstract}

\maketitle

\section{Introduction}
Let $\Aa$ be a stationary and ergodic random coefficient field on $\R^d$, constructed on some probability space $(\Omega,\Pm)$, that satisfies the boundedness and ellipticity properties
\begin{align}\label{eq:unif-ell}
|\Aa(x,\omega)\ee|\le|\ee|,\qquad\ee\cdot\Aa(x,\omega)\ee\ge\lambda|\ee|^2,\qquad\text{for all $x,\ee\in\R^d$ and $\omega\in\Omega$},
\end{align}
for some $\lambda>0$.
For all deterministic vector fields $f\in\Ld^2(\R^d)^d$ and $\omega\in\Omega$, we consider the unique Lax-Milgram solution $u_f(\cdot,\omega)\in\dot H^1(\R^d)$ of the following elliptic PDE in $\R^d$,
\[-\nabla\cdot\Aa(\cdot,\omega) \nabla u_f(\cdot,\omega)=\nabla\cdot f.\]
The solution operator (or Helmholtz projection) $\Hc:=\nabla(\nabla\cdot\Aa\nabla)^{-1}\nabla\cdot:f\mapsto-\nabla u_f$ is then a bounded operator $\Ld^2(\R^d)\to\Ld^2(\R^d\times\Omega)$. In this note, we aim at studying the average $\nabla\expec{u_f}=-\expec{\Hc}f$ of the solution operator with respect to the underlying ensemble of coefficient fields --- a problem which seems to have been set aside so far in the homogenization community
and is particularly relevant in the setting of fluctuations (cf.\@ Section~\ref{sec:fluc}).
The following straightforward lemma elucidates the structure of this averaged solution operator; a short proof is included in Appendix~\ref{app}.

\begin{lem}\label{lem:main}
With the above notation, there is a unique self-adjoint convolution operator~$B$ on $\Ld^2(\R^d)$
that satisfies $-\lambda\triangle \le B \le -\triangle$ and such that
for all $f\in\Ld^2(\R^d)^d$ the averaged solution $\expec{u_f}\in\dot H^1(\R^d)$ is the unique Lax-Milgram solution of
\[B\,\expec{u_f}=\nabla\cdot f.\qedhere\]
\end{lem}

This motivates a detailed study of the properties of the Fourier symbol $\hat B$.
In view of homogenization regimes, we are particularly interested in the regularity of $\hat B$ at the origin.
Following a preliminary work by Sigal~\cite{Sigal}, a recent result by Bourgain~\cite{Bourgain-18} (in the nearly-optimal version due to Kim and the third author~\cite{Lemm-18}) solves this problem in the model framework of discrete equations with iid coefficients, in the perturbative regime of a small ellipticity contrast.

\begin{theor}[\cite{Bourgain-18,Lemm-18}]\label{th:Bmain}
Let $d\ge3$.
Consider a random coefficient field $\Aa^\delta$ on $\Z^d$ given by
$\Aa^\delta(x)=\Id-\delta\,\bb(x)$,
where $\delta>0$ and where $\{\bb(x)\}_{x\in\Z^d}$ is a family of real-valued iid random matrices with $|\bb(x)|\le1$. For all $f\in\Ld^2(\Z^d)^d$ and $\omega\in\Omega$, we denote by $u_f^\delta(\cdot,\omega)\in\dot H^1(\Z^d)$ the unique Lax-Milgram solution of the following discrete equation in $\Z^d$,\footnote{In this statement, $\nabla$ denotes the discrete gradient, defined componentwise by $\nabla_ju(x):=u(x+\ee_j)-u(x)$ for the $j$-th standard unit vector $\ee_j$, and $-\nabla^*\cdot$ is its formal adjoint.}
\[-\nabla^*\cdot\Aa^\delta(\cdot,\omega)\nabla u_f^\delta(\cdot,\omega)=\nabla^*\cdot f,\]
and we consider the convolution operator $B^\delta$ on $\Ld^2(\Z^d)$ such that for all $f\in\Ld^2(\Z^d)^d$ the averaged solution $\expecm{u_{f}^\delta}\in\dot H^1(\Z^d)$ is the unique Lax-Milgram solution of
\[B^\delta\,\expecm{u_{f}^\delta}=\nabla^*\cdot f.\]
Then $B^\delta$ can be written as $B^\delta=-\nabla^* \cdot B_0^\delta\nabla$ for some convolution operator $B_0^\delta$ on $\Ld^2(\Z^d)^d$ and there is a universal constant $C_d>0$ such that
for all $0<\delta<\tfrac1{C_d}$ the Fourier symbol $\hat B_0^\delta$ is of Hölder class $C^{2d-C_d\delta}$.
\end{theor}

A natural conjecture concerns the same regularity for the symbol $\hat B$ beyond the small ellipticity ratio regime $\delta\ll1$ and under general mixing conditions (rather than in the iid case). We focus for simplicity on the continuum setting.

\begin{conj}[Bourgain \& Spencer]\label{conj:main}
Under suitable mixing conditions on the random coefficient field $\Aa$, the operator $B$ defined in Lemma~\ref{lem:main} can be written as $B=-\nabla\cdot B_0\nabla$ for some convolution operator $B_0$ on $\Ld^2(\R^d)^d$ such that the Fourier symbol $\hat B_0$ is of Hölder class $C^{2d-\eta}$ at the origin for all $\eta>0$.
\end{conj}

In the sequel, we discuss how such a regularity result is to be interpreted in the framework of homogenization: a higher regularity of $\hat B$ at the origin is equivalent to obtaining a higher-order approximation
of the averaged solution operator $\expec{\nabla u_{\e,f}}$ in the homogenization regime.
In particular, we show that the derivatives of the symbol $\hat B$ at the origin provide an alternative definition of the (symmetrized) higher-order homogenized coefficients.
While the classical ($\Ld^2$-based) corrector theory in stochastic homogenization only allows to define homogenized coefficients up to order $\le\lceil\frac d2\rceil$, the above conjectured regularity of $\hat B$ would allow to proceed up to order $\le2d$. This comes along with a higher-order description of the averaged solution beyond the accuracy allowed by the classical corrector theory.
In fact, while the classical corrector theory is optimal in view of the strong effective approximation of the solution operator in $\Ld^2(\R^d\times\Omega)$, the results described here beg for the development of a novel higher-order corrector theory in a weak sense in probability. This shares some close connection with results in~\cite{DS1}, and the investigation of Conjecture~\ref{conj:main} in this spirit is postponed to a future work.

To further illustrate the above relation between homogenized coefficients and regularity of $\hat B_0$, we 
also consider the case of a periodic coefficient field~$\Aa$: we then prove that $\hat B_0$ is analytic at the origin, which is equivalent to the well-known existence and exponential boundedness of all homogenized coefficients.

\section{Regularity of $\hat B_0$ and homogenization}

In this section, we establish the following general result stating the equivalence between the regularity of the symbol $\hat B_0$ at the origin and the higher-order description of the averaged solution operator in the homogenization regime.
Note that only the {\it symmetrized} higher-order homogenized coefficients are characterized.
In what follows  $(\cdot)^T$ stands for matrix transposition.

\begin{prop}\label{prop:main-gen}
Let $d\ge1$. Given regularity exponents $\ell\in\N$ and $0<\eta<1$, the following two properties are equivalent:
\begin{enumerate}[\qquad\emph{(I)}]
\item[\emph{(I)}]
The operator $B$ defined in Lemma~\ref{lem:main} can be written as $B=-\nabla\cdot B_0\nabla$ for some convolution operator $B_0$ on $\Ld^2(\R^d)^d$ such that the Fourier symbol $\hat B_0$ is of Hölder class $C^{\ell-\eta}$ at the origin.\footnote{This is understood in the sense of $\big|\hat B_{0}(\xi)-\sum_{\alpha\in\N^d\atop|\alpha|\le\ell-1}\frac{\nabla^\alpha\hat B_{0}(0)}{\alpha!}\,\xi^{\alpha}\big|\le C_\ell |\xi|^{\ell-\eta}$ for all $\xi\in\R^d$.}
\item[\emph{(II)}]
There exist  ``higher-order homogenized coefficients'' $(\bar\Aa^n)_{1\le n\le\ell}$ with $\lambda\Id\le\frac12(\bar\Aa^1+(\bar \Aa^1)^T)\le\frac1\lambda\Id$ and with the following property:
For all $\e>0$, $f\in\Ld^2(\R^d)^d$, and $\omega\in\Omega$, defining $u_{\e,f}(\cdot,\omega)\in\dot H^1(\R^d)$ as the unique Lax-Milgram solution of the following rescaled elliptic PDE in $\R^d$,
\begin{align}\label{eq:ueps}
\qquad-\nabla\cdot\Aa(\tfrac\cdot\e,\omega)\nabla u_{\e,f}(\cdot,\omega)=\nabla\cdot f,
\end{align}
and defining the ``$\ell$th-order homogenized solution'' $\bar u_{\e,f}^\ell:=\sum_{n=1}^\ell\e^{n-1}\tilde u_f^n$ where $\tilde u_f^1\in\dot H^1(\R^d)$ denotes the unique Lax-Milgram solution of
\begin{equation}\label{ant1}
\qquad-\nabla\cdot\bar\Aa^1\nabla\tilde u_f^1=\nabla\cdot f,
\end{equation}
and where for $2\le n\le\ell$ we inductively define $\tilde u_f^n$ as the unique Lax-Milgram solution of 
\begin{equation}\label{ant2}
\qquad-\nabla\cdot\bar\Aa^1\nabla\tilde u^n_f=\nabla\cdot\Big(\sum_{k=2}^{n}\bar\Aa^k_{i_1\ldots i_{k-1}}\nabla\nabla^{k-1}_{i_1\ldots i_{k-1}}\tilde u_f^{n+1-k}\Big),
\end{equation}
(with the Einstein summation convention on the repeated indices $i_1,\dots,i_{k-1}$)
there holds
\begin{align}\label{eq:est-gen-aver}
\qquad\big\|\nabla (\expec{u_{\e,f}}-\bar u_{\e,f}^\ell)\big\|_{\Ld^2(\R^d)}\,\le\,\e^{\ell-\eta}\,C_\ell\,\|\langle\nabla\rangle^{2\ell-1}f\|_{\Ld^2(\R^d)},
\end{align}
for some constant $C_\ell$ only depending on $d,\lambda,\ell$, where $\langle\nabla\rangle$ has Fourier multiplier $\sqrt{1+|\xi|^2}$.
\end{enumerate}
The (symmetrized) higher-order homogenized coefficients are then related to the derivatives of the symbol~$\hat B_0$ at the origin via the following formulas: for all $1\le n\le\ell$,
\begin{equation}\label{eq:link-B-a}
\qquad\sum_{\alpha\in\N^d\atop|\alpha|=n-1}\frac{\nabla^\alpha\hat B_0(0)}{i^\alpha \alpha!}\,\xi^{\alpha}\,=\,\tfrac12\big(\bar\Aa_{i_1\ldots i_{n-1}}^{n}+(\bar\Aa_{i_1\ldots i_{n-1}}^{n})^T\big)\,\xi_{i_1}\ldots\xi_{i_{n-1}},
\end{equation}
an identity between square symmetric matrices.
In addition, \emph{(I)} holds with $\hat B_0$ analytic at the origin if and only if \emph{(II)} holds for all $\ell\ge1$ with $C_\ell=C^\ell$ in~\eqref{eq:est-gen-aver} and with $|\bar\Aa^\ell|\le C^\ell$, for some constant $C$ only depending on $d,\lambda$.
\end{prop}

\begin{rem}\label{rem:homog-eq}
While standard two-scale expansion techniques would rather suggest to define the $\ell$th-order homogenized solution as satisfying
\[-\nabla\cdot\Big(\sum_{k=1}^\ell\e^{k-1}\bar\Aa^k_{i_1\ldots i_{k-1}}\nabla^{k-1}_{i_1\ldots i_{k-1}}\Big)\nabla\bar U_{\e,f}^\ell=\nabla\cdot f,\]
we note that this equation is ill-posed in general (the Fourier symbol of the differential operator may vanish) and the above definition of $\bar u_{\e,f}^\ell$ precisely provides a well-defined proxy (cf.\@ also~\cite{KMS-06}).
\end{rem}

\begin{proof}[Proof of Proposition~\ref{prop:main-gen}]
We split the proof into two steps, first showing that (I) implies~(II) and then turning to the converse.

\medskip
\step1 (I) implies (II).
\nopagebreak

\noindent
Given $B_0$ as in~(I), let the (symmetrized) coefficients~$(\bar\Aa^n)_{n=1}^\ell$ be defined by~\eqref{eq:link-B-a} and let $\bar u_{\e,f}^\ell$ be as in property~(II).
We first examine the equation satisfied by $\bar u_{\e,f}^\ell$.
Summing the defining equations for $(\tilde u_f^n)_{1\le n\le\ell}$, we find
\[-\nabla\cdot\bar\Aa^1\nabla\bar u^\ell_{\e,f}=\nabla\cdot f+\nabla\cdot\bigg(\sum_{k=2}^\ell\bar\Aa^k_{i_1\ldots i_{k-1}}\nabla\nabla^{k-1}_{i_1\ldots i_{k-1}}\sum_{n=k}^\ell\e^{n-1}\tilde u_f^{n+1-k}\bigg),\]
or equivalently, reorganizing this identity,
\begin{multline}\label{eq:barun}
-\nabla\cdot\Big(\sum_{k=1}^\ell\e^{k-1}\bar\Aa^k_{i_1\ldots i_{k-1}}\nabla^{k-1}_{i_1\ldots i_{k-1}}\Big)\nabla\bar u^\ell_{\e,f}\\
=\nabla\cdot f-\nabla\cdot\bigg(\sum_{k=2}^\ell\e^{k-1}\bar\Aa^k_{i_1\ldots i_{k-1}}\nabla\nabla^{k-1}_{i_1\ldots i_{k-1}}\sum_{n=\ell+2-k}^\ell\e^{n-1}\tilde u_f^{n}\bigg).
\end{multline}
By $\e$-rescaling, $\expec{u_{\e,f}}$ satisfies $-\nabla\cdot B_{\e,0}\nabla\expec{u_{\e,f}}=\nabla\cdot f$ with symbol 
$$
\hat B_{\e,0}(\xi)=\hat B_0(\e\xi)
$$
(this crucial identity reflects that homogenization takes place at large scales, or equivalently, 
 at low frequencies).
Injecting this into the above and using the definition~\eqref{eq:link-B-a} of the coefficients $(\bar\Aa^n)_{1\le n\le\ell}$, we obtain
\begin{multline*}
-\nabla\cdot B_{\e,0}\nabla(\expec{u_{\e,f}}-\bar u_{\e,f}^\ell)=
\nabla\cdot\bigg(B_{\e,0}-\sum_{\alpha\in\N^d\atop|\alpha|\le\ell-1}\frac{\nabla^\alpha\hat B_{\e,0}(0)}{i^\alpha \alpha!}\,\nabla^{\alpha}\bigg)\nabla\bar u^\ell_{\e,f}\\
+\nabla\cdot\bigg(\sum_{k=2}^\ell\e^{k-1}\bar\Aa^k_{i_1\ldots i_{k-1}}\nabla\nabla^{k-1}_{i_1\ldots i_{k-1}}\sum_{n=\ell+2-k}^\ell\e^{n-1}\tilde u_f^{n}\bigg).
\end{multline*}
Using the regularity of $\hat B_0$ (cf.~(I)) in the form
\begin{equation*}
\bigg|\hat B_{\e,0}(\xi)-\sum_{\alpha\in\N^d\atop|\alpha|\le\ell-1}\frac{ \nabla^\alpha\hat B_{\e,0}(0)}{\alpha!}\,\xi^{\alpha}\bigg|\le(\e|\xi|)^{\ell-\eta}C_\ell,
\end{equation*}
for some constant $C_\ell$, an energy estimate then yields
\begin{align*}
\big\|\nabla(\expec{u_{\e,f}}-\bar u_{\e,f}^\ell)\big\|_{\Ld^2(\R^d)}\le \e^{\ell-\eta}\,C_\ell\sum_{n=1}^\ell\|\langle\nabla\rangle^{\ell}\nabla\tilde u^n_{f}\|_{\Ld^2(\R^d)}\le \e^{\ell-\eta}\,C_\ell\,\|\langle\nabla\rangle^{2\ell-1}f\|_{\Ld^2(\R^d)},
\end{align*}
and property~(II) follows for some other constant $C_\ell$.

\medskip
\step2 (II) implies (I).

\nopagebreak 
\noindent
In this part of the proof, we use the following slight abuse of notation: $C_\ell$ denotes a constant that 
might differ from that of the assumption (II) by a multiplicative factor that only depends on $d$ and on the ellipticity contrast $\lambda$
--- in particular, it may change from line to line.
For all $\xi \in \R^d$, set 
\begin{equation}\label{add2}
\hat B_{0,\e}^\ell(\xi):=\sum_{k=1}^\ell\e^{k-1}\bar\Aa^k_{i_1\ldots i_{k-1}}\,(i\xi_{i_1})\ldots(i\xi_{i_{k-1}}).
\end{equation}
For $\e\le\frac1{C_\ell}$ small enough, it follows from the bound $\lambda\Id\le \bar\Aa^1\le\frac1\lambda\Id$ together with
the finiteness of the $\bar \Aa^k$'s that for all $|\xi|\le1$ and $|e|=1$,
\begin{align}\label{eq:Re-sumbar}
\tfrac\lambda2~\le~|e\cdot \hat B_{0,\e}^\ell(\xi)e|~\le~\tfrac2\lambda.
\end{align}
Equation~\eqref{eq:barun} for $\bar u^\ell_{\e,f}$ can then be inverted in Fourier space:
the Fourier transform $\F\nabla\bar u^\ell_{\e,f}$ of $\nabla\bar u^\ell_{\e,f}$ is given by
\begin{multline*}
(-\F\nabla\bar u^\ell_{\e,f})(\xi)=\frac{\xi\otimes\xi}{\xi\cdot\hat B_{0,\e}^\ell(\xi)\xi}\\
\cdot\bigg(\hat f(\xi)+\sum_{k=2}^\ell\e^{k-1}\bar\Aa^k_{i_1\ldots i_{k-1}}(i\xi_{i_1})\ldots(i\xi_{i_{k-1}})\sum_{n=\ell+2-k}^{\ell}\e^{n-1}(-\F\nabla\tilde u^n_f)(\xi)\bigg).
\end{multline*}
Using \eqref{ant1} and \eqref{ant2} in Fourier space to express $\F\nabla\tilde u^n_f(\xi)$ in terms of $\hat f(\xi)$, this yields
\begin{align*}
\bigg|(-\F\nabla\bar u^\ell_{\e,f})(\xi)-\frac{\xi\otimes\xi}{\xi\cdot\hat B_{0,\e}^\ell(\xi)\xi}\cdot\hat f(\xi)\bigg|
\,\le\, \e^\ell C_\ell\langle\xi\rangle^{2\ell-2}|\hat f(\xi)|.
\end{align*}
Combining the latter with~\eqref{eq:est-gen-aver} and with the equation $-\nabla\cdot B_{\e,0}\nabla\expec{u_{\e,f}}=\nabla\cdot f$ with symbol $\hat B_{\e,0}(\xi)=\hat B_0(\e\xi)$, we obtain
\begin{align*}
\Bigg(\int_{\R^d}\bigg|\Big(\frac{\xi \otimes\xi}{\xi\cdot \hat B_0(\e\xi)\xi}-\frac{\xi \otimes\xi}{\xi\cdot\hat B_{0,\e}^\ell(\xi)\xi}\Big)\cdot\hat f(\xi)\bigg|^2d\xi\Bigg)^\frac12
\,\le\, \e^{\ell-\eta} C_\ell\|\langle\nabla\rangle^{2\ell-1}f\|_{\Ld^2(\R^d)}.
\end{align*}
By~\eqref{eq:Re-sumbar} and the a priori bound $\hat B(\xi)\le|\xi|^2$, 
we reformulate the integrand for all $|\xi|\le 1$ as
\begin{eqnarray*}
\lefteqn{\bigg|\Big(\frac{\xi \otimes\xi}{\xi\cdot \hat B_0(\e\xi)\xi}-\frac{\xi \otimes\xi}{\xi\cdot\hat B_{0,\e}^\ell(\xi)\xi}\Big)\cdot\hat f(\xi)\bigg|^2}
\\
&=&|\xi|^4 \Big|\hat f(\xi)\cdot \frac{\xi}{|\xi|}\Big|^2\frac{\big|\xi\cdot \hat B_0(\e\xi)\xi-\xi\cdot\hat B_{0,\e}^\ell(\xi)\xi\big|^2}{\big|\xi \cdot \hat B_0(\e\xi)\xi\big|^2\big|\xi\cdot\hat B_{0,\e}^\ell(\xi)\xi\big|^2}
\\
&\ge& \frac{\lambda^2}{4}\Big|\hat f(\xi)\cdot \frac{\xi}{|\xi|}\Big|^2\big|\xi \cdot \hat B_0(\e\xi)\xi -\xi\cdot\hat B_{0,\e}^\ell(\xi)\xi\big|^2.
\end{eqnarray*}
Since the function $f\in\Ld^2(\R^d)^d$ is arbitrary, we deduce for almost all $|\xi|\le1$,
\begin{align}\label{add3}
\big|\xi \cdot \hat B_0(\e\xi)\xi-\xi\cdot\hat B_{0,\e}^\ell(\xi)\xi\big|
\,\le\, \e^{\ell-\eta} C_\ell,
\end{align}
and property~(I) follows.
Finally, the formula  \eqref{eq:link-B-a} follows from the combination of  \eqref{add2} and \eqref{add3}.
\end{proof}

\section{Periodic setting}
In this section, we consider the particular case of a periodic coefficient field $\Aa$ on $\R^d$ satisfying the boundedness and ellipticity properties~\eqref{eq:unif-ell}.
More precisely, we consider the ensemble of coefficient fields $\{\Aa(\cdot+z)\}_{z\in Q}$, where $Q=[-\frac12,\frac12)^d$ is the periodicity cell, and the ensemble average is then with respect to the Lebesgue measure for translations~$z\in Q$. The probability space $(\Omega,\Pm)$ in the introduction thus reduces to the cell $Q$ endowed with the Lebesgue measure.
In this setting, using the classical corrector theory, we show that property~(II) in Proposition~\ref{prop:main-gen} is satisfied for all $\ell\ge1$. In terms of the symbol $\hat B$, our main result then takes on the following guise.

\begin{theor}\label{th:per}
Let $\Aa$ be periodic.
The operator $B$ defined in Lemma~\ref{lem:main} can be written as $B=-\nabla\cdot B_0\nabla$ for some convolution operator $B_0$ on $\Ld^2(\R^d)^d$ such that the Fourier symbol $\hat B_0$ is analytic in a neighborhood of the origin. In addition, for all $n\ge1$, the usual $n$th-order homogenized coefficients $\bar \Aa^{n}$ (cf.~\eqref{e.cor-2}) is related to the $(n-1)$th gradient $\nabla^{n-1}\hat B_0(0)$ via formula~\eqref{eq:link-B-a}.
\end{theor}

We start with recalling the classical inductive definition of the higher-order correctors $(\varphi^n)_{n\ge0}$, homogenized coefficients $(\bar\Aa^n)_{n\ge1}$, fluxes $(q^n)_{n\ge1}$, and flux correctors $(\sigma^n)_{n\ge0}$ in periodic homogenization (cf.~\cite{BLP-78,JKO94}).

\begin{enumerate}[$\bullet$]
\item $\varphi^0 \equiv 1$ and for all $n\ge1$ we define $\varphi^n:=(\varphi^n_{i_1\ldots i_n})_{1\le i_1,\ldots,i_n\le d}$ with $\varphi^n_{i_1\ldots i_n}\in\Ld^2_\per(Q)$ the periodic scalar field satisfying
\begin{gather}\label{e.cor-1}
-\nabla\cdot\Aa\nabla\varphi^n_{i_1\ldots i_n}=\nabla\cdot\big((\Aa\varphi^{n-1}_{i_1\ldots i_{n-1}}-\sigma^{n-1}_{i_1\ldots i_{n-1}})\,\ee_{i_n}\big),
\end{gather}
with $\int_Q\varphi^n_{i_1\ldots i_n}=0$.
\item For all $n\ge1$ we define $\bar\Aa^n:=(\bar\Aa^n_{i_1\ldots i_{n-1}})_{1\le i_1,\ldots,i_{n-1}\le d}$ with $\bar\Aa^n_{i_1\ldots i_{n-1}}\in\R^{d\times d}$ given by
\begin{equation}\label{e.cor-2}
\bar\Aa^n_{i_1\ldots i_{n-1}}\ee_j:=\int_Q\Aa\big(\nabla\varphi^{n}_{i_1\ldots i_{n-1}j}+\varphi^{n-1}_{i_1\ldots i_{n-1}}\ee_{j}\big).
\end{equation}
\item For all $n\ge1$ we define $q^n:=(q^n_{i_1\ldots i_n})_{1\le i_1,\ldots,i_n\le d}$ with $q^n_{i_1\ldots i_n}\in\Ld^2_\per(Q)^d$ the periodic vector field given by
\begin{equation}\label{e.cor-3}
q^n_{i_1\ldots i_n}:=\Aa\nabla\varphi^{n}_{i_1\ldots i_n}+(\Aa\varphi^{n-1}_{i_1\ldots i_{n-1}}-\sigma^{n-1}_{i_1\ldots i_{n-1}})\,\ee_{i_n}-\bar\Aa_{i_1\ldots i_{n-1}}^{n}\ee_{i_n},
\end{equation}
where the definition of $\bar\Aa^{n}$ ensures $\int_Qq^n=0$.
\item $\sigma^0\equiv 0$ and for all $1\le n\le\ell$ we define $\sigma^n:=(\sigma^n_{i_1\ldots i_n})_{i_1,\ldots,i_n=1}^d$ with $\sigma^n_{i_1\ldots i_n}\in\Ld^2_\per(Q)^{d\times d}$ the periodic skew-symmetric matrix field satisfying
\begin{gather}\label{e.cor-4}
-\triangle\sigma^n_{i_1\ldots i_n}=\nabla\times q^n_{i_1\ldots i_n},\qquad\nabla\cdot\sigma^n_{i_1\ldots i_n}=q^n_{i_1\ldots i_n},
\end{gather}
with $\int_Q\sigma^n_{i_1\ldots i_n}=0$, with the notation $(\nabla\times X)_{ij}:=\nabla_iX_j-\nabla_jX_i$ for a vector field $X$ and with the notation $(\nabla\cdot Y)_i:=\nabla_jY_{ij}$ for a matrix field $Y$.
\end{enumerate}

\medskip
An iterative use of the Poincaré inequality on $Q$ yields the following, which ensures the well-posedness of the above objects and provides a priori bounds.

\begin{lem}[Periodic correctors]\label{lem:cor-per}
Let $d\ge1$ and let the coefficient field $\Aa$ be periodic and satisfy~\eqref{eq:unif-ell}.
Then the above collections $(\varphi^n,\sigma^n)_{n\ge0}$, $(\bar\Aa^n)_{n\ge1}$, and $(q^n)_{n\ge1}$ are uniquely defined and satisfy for all $n\ge1$,
\begin{gather*}
\|(\varphi^n,\sigma^n)\|_{\Ld^2(Q)}+|\bar\Aa^n|+\|q^n\|_{\Ld^2(Q)}\,\le\,C^n,
\end{gather*}
where the constant $C$ depends only on $d,\lambda$.
\end{lem}

\begin{proof}
A priori estimates yield for all $n\ge1$
\begin{eqnarray*}
\lambda\|\nabla\varphi^n\|_{\Ld^2(Q)}&\lesssim&\|\varphi^{n-1}\|_{\Ld^2(Q)}+\|\sigma^{n-1}\|_{\Ld^2(Q)},\\
\|\nabla\sigma^n\|_{\Ld^2(Q)}&\lesssim&\|q^n\|_{\Ld^2(Q)},\\
\|q^n\|_{\Ld^2(Q)}&\lesssim&\|\nabla\varphi^n\|_{\Ld^2(Q)}+\|\varphi^{n-1}\|_{\Ld^2(Q)}+\|\sigma^{n-1}\|_{\Ld^2(Q)}+|\bar\Aa^{n}|,\\
|\bar\Aa^n|&\lesssim&\|\nabla\varphi^n\|_{\Ld^2(Q)}+\|\varphi^{n-1}\|_{\Ld^2(Q)}.
\end{eqnarray*}
Applying the Poincaré inequality, the conclusion follows from a direct induction.
\end{proof}

We recall the use of these correctors in homogenization.
For $f\in\Ld^2(\R^d)^d$, we consider the solution $u_{\e,f}$ of the rescaled elliptic PDE~\eqref{eq:ueps}.
Standard two-scale expansion techniques~\cite{BLP-78} suggest the Ansatz
\begin{align}\label{eq:approx-2scale}
\nabla u_{\e,f}=\sum_{k=0}^\infty\e^k\varphi^k_{i_1\ldots i_k}(\tfrac\cdot\e)\,\nabla^k_{i_1\ldots i_k}\bar U_{\e,f},
\end{align}
where $\bar U_{\e,f}$ satisfies
\[-\nabla\cdot\Big(\sum_{k=1}^\infty\e^k\bar\Aa^k_{i_1\ldots i_{k-1}}\nabla^{k-1}_{i_1\ldots i_{k-1}}\Big)\nabla\bar U_{\e,f}=\nabla\cdot f.\]
Since the convergence of the series in~\eqref{eq:approx-2scale} does not hold in general for $f\in\Ld^2(\R^d)^d$, we focus on partial sum approximations. Moreover, as in Remark~\ref{rem:homog-eq}, the equation for $\bar U_{\e,f}$ is ill-posed in general and a suitable proxy needs to be devised (cf.\@ also~\cite{KMS-06}).
A precise statement is as follows; note that Theorem~\ref{th:per} is then a consequence of the equivalence in Proposition~\ref{prop:main-gen}.

\begin{prop}[Classical corrector theory --- periodic setting]\label{prop:cor-per}
Let $d\ge1$ and let the coefficient field $\Aa$ be periodic.
Given $f\in C^\infty_c(\R^d)^d$ and $n\ge1$, let the $n$th-order homogenized solution $\bar u_{\e,f}^n$ for~\eqref{eq:ueps} be defined as in the statement of Proposition~\ref{prop:main-gen}(II) with homogenized coefficients defined in~\eqref{e.cor-2}.
Then, for all $n\ge0$,
\begin{align*}
\Big\|\nabla \Big(u_{\e,f}-\sum_{k=0}^n\e^{k}\varphi_{i_1\ldots i_k}^k(\tfrac\cdot\e)\nabla_{i_1\ldots i_k}^k\bar u_{\e,f}^n\Big)\Big\|_{\Ld^2(\R^d\times Q)}\,\le\,(\e C)^n\|\langle\nabla\rangle^{2n-1}f\|_{\Ld^2(\R^d)},
\end{align*}
where the constant $C$ depends only on $d,\lambda$. In particular, for all $n\ge0$,
\[\big\|\nabla(\expec{u_{\e,f}}-\bar u_{\e,f}^n)\big\|_{\Ld^2(\R^d)}\,\le\,(\e C)^n\|\langle\nabla\rangle^{2n-1}f\|_{\Ld^2(\R^d)}.\qedhere\]
\end{prop}

\begin{rem}
With the definition $\{\Aa(x,z):=\Aa(x+z)\}_{z\in Q}$ of the periodic ensemble of coefficient fields, recall that the solution $u_{\e,f}$ is viewed as a map $\R^d\times Q\to\R$, where for a translation $z\in Q$ the function $u_{\e,f}(\cdot,z)\in\dot H^1(\R^d)$ is the unique Lax-Milgram solution in~$\R^d$ of
\[-\nabla \cdot \Aa(\tfrac\cdot\e+z) \nabla u_{\e,f}(\cdot,z)\,=\,\nabla \cdot f.\]
The averaged solution then takes the form $\expec{u_{\e,f}}(x):=\int_Q u_{\e,f}(x,z)\,dz$.
\end{rem}

\begin{proof}
By scaling, it suffices to consider $\e=1$, and we drop it from all subscripts in the notation.
We split the proof into two steps.

\medskip
\step1 For $n\ge0$, given $\bar w\in C^\infty_c(\R^d)$, we define its \emph{$n$th-order two-scale expansion}
\[F^n[\bar w]:=\sum_{k=0}^n\varphi_{i_1\ldots i_k}^k\nabla_{i_1\ldots i_k}^k\bar w,\]
and we claim that it satisfies the following PDE in $\R^d$,
\begin{multline}\label{eq:Fnw}
\nabla\cdot\Aa\nabla F^n[\bar w]
\,=\,\nabla\cdot\Big(\sum_{k=1}^{n}\bar\Aa^k_{i_1\ldots i_{k-1}}\nabla^{k-1}_{i_1\ldots i_{k-1}}\Big)\nabla\bar w\\
+\nabla\cdot\big((\Aa\varphi_{i_1\ldots i_n}^n-\sigma_{i_1\ldots i_n}^n)\nabla\nabla^n_{i_1\ldots i_n}\bar w\big).
\end{multline}
A proof can be found e.g.~in~\cite{DO1}: it follows from an inductive computation, exploiting the definition of correctors and flux correctors.
It is reproduced here for completeness.
The claim~\eqref{eq:Fnw} is obvious for $n=0$. Now, if it holds for some~$n\ge0$, we deduce
\begin{multline}\label{eq:pre-dec-homog-induc}
\nabla\cdot\Aa\nabla F^{n+1}[\bar w]
\,=\,\nabla\cdot\bigg(\sum_{k=1}^{n}\bar\Aa^{k}_{i_1\ldots i_{k-1}}\nabla\nabla^{k-1}_{i_1\ldots i_{k-1}}\bar w\bigg)\\
+\nabla\cdot\big((\Aa\varphi_{i_1\ldots i_n}^n-\sigma_{i_1\ldots i_n}^n)\nabla\nabla^n_{i_1\ldots i_n}\bar w\big)+\nabla\cdot\Aa\nabla\big(\varphi_{i_1\ldots i_{n+1}}^{n+1}\nabla_{i_1\ldots i_{n+1}}^{n+1}\bar w\big).
\end{multline}
The definition of $\sigma_{i_1\ldots i_{n+1}}^{n+1}$ yields
\begin{multline*}
\nabla\cdot\big((\Aa\varphi_{i_1\ldots i_n}^n-\sigma_{i_1\ldots i_n}^n)\nabla\nabla^n_{i_1\ldots i_n}\bar w\big)
\,=\,\nabla\cdot\big((\nabla\cdot\sigma_{i_1\ldots i_{n+1}}^{n+1}) \nabla^{n+1}_{i_1\ldots i_{n+1}}\bar w\big)\\
-\nabla\cdot\big(\Aa\nabla\varphi_{i_1\ldots i_{n+1}}^{n+1}\nabla^{n+1}_{i_1\ldots i_{n+1}}\bar w\big)+\nabla\cdot\big(\bar\Aa_{i_1\ldots i_n}^{n+1}\nabla\nabla^{n}_{i_1\ldots i_{n}}\bar w\big),
\end{multline*}
and hence, using the skew-symmetry of $\sigma_{i_1\ldots i_{n+1}}^{n+1}$ and decomposing
\[\nabla\varphi_{i_1\ldots i_{n+1}}^{n+1}\nabla^{n+1}_{i_1\ldots i_{n+1}}\bar w=\nabla(\varphi_{i_1\ldots i_{n+1}}^{n+1}\nabla^{n+1}_{i_1\ldots i_{n+1}}\bar w)-\varphi_{i_1\ldots i_{n+1}}^{n+1}\nabla\nabla^{n+1}_{i_1\ldots i_{n+1}}\bar w,\]
we obtain
\begin{multline*}
\nabla\cdot\big((\Aa\varphi_{i_1\ldots i_n}^n-\sigma_{i_1\ldots i_n}^n)\nabla\nabla^n_{i_1\ldots i_n}\bar w\big)
\,=\,\nabla\cdot\big((\Aa\varphi_{i_1\ldots i_{n+1}}^{n+1}-\sigma_{i_1\ldots i_{n+1}}^{n+1})\nabla\nabla^{n+1}_{i_1\ldots i_{n+1}}\bar w\big)\\
-\nabla\cdot\Aa\nabla\big(\varphi_{i_1\ldots i_{n+1}}^{n+1}\nabla^{n+1}_{i_1\ldots i_{n+1}}\bar w\big)+\nabla\cdot\big(\bar\Aa_{i_1\ldots i_n}^{n+1}\nabla\nabla^{n}_{i_1\ldots i_{n}}\bar w\big).
\end{multline*}
Injecting this into~\eqref{eq:pre-dec-homog-induc} leads to the claim~\eqref{eq:Fnw} at level $n+1$.

\medskip
\step2 Conclusion.

\nopagebreak
\noindent
Let $n\ge1$.
Combining~\eqref{eq:Fnw} with the equation~\eqref{eq:barun} for $\bar u_{f}^n$ leads to
\begin{multline*}
-\nabla\cdot\Aa\nabla (u_f-F^n[\bar u_f^n])
\,=\,\nabla\cdot\bigg(\sum_{k=2}^{n}\bar\Aa^k_{i_1\ldots i_{k-1}}\nabla^{k-1}_{i_1\ldots i_{k-1}}\nabla\sum_{l=n+2-k}^{n}\tilde u^{l}\bigg)\\
+\nabla\cdot\big((\Aa\varphi_{i_1\ldots i_n}^n-\sigma_{i_1\ldots i_n}^n)\nabla\nabla^n_{i_1\ldots i_n}\bar u_f^n\big).
\end{multline*}
The a priori estimates of Lemma~\ref{lem:cor-per} yield
\begin{eqnarray*}
\lefteqn{\int_{\R^d}\int_Q \big|(\Aa\varphi_{i_1\ldots i_n}^n-\sigma_{i_1\ldots i_n}^n)(x,z)\nabla\nabla^n_{i_1\ldots i_n}\bar u_f^n(x)\big|^2dz\,dx}
\\
&\le& \int_{\R^d}|\nabla^{n+1}\bar u_f^n(x)|^2 \int_Q |(\Aa\varphi_{i_1\ldots i_n}^n-\sigma_{i_1\ldots i_n}^n)(x+z)|^2dz\,dx
\\
&\le& C^n\|\nabla^{n+1}\bar u_f^n\|_{\Ld^2(\R^d)}^2.
\end{eqnarray*}
Hence, by an energy estimate,
\begin{align*}
\|\nabla (u_f-F^n[\bar u_f^n])\|_{\Ld^2(\R^d\times Q)}\,&\le\, C^n\sum_{k=2}^{n}\sum_{l=n+2-k}^{n}\|\nabla^{k}\tilde u^{l}\|_{\Ld^2(\R^d)}+C^n\|\nabla^{n+1}\bar u_f^n\|_{\Ld^2(\R^d)}\\
\,&\le\,C^{n}\|\nabla^n\langle\nabla\rangle^{n-1}f\|_{\Ld^2(\R^d)}.\qedhere
\end{align*}
\end{proof}

\section{Random setting}

In this section, we follow the approach presented in the periodic setting and start by recalling the conclusions of the classical corrector theory. 
Under sufficient mixing conditions on the coefficient field $\Aa$, correctors and flux correctors $(\varphi^n,\sigma^n)$ are now well-defined in $\Ld^2(\Omega)$ only up to order $n<\lceil\frac{d}2\rceil$.
As a consequence, we obtain an analogue of Proposition~\ref{prop:cor-per} for the $\Ld^2$-approximation of the solution $u_{\e,f}$ up to the accuracy~$O(\e^{d/2})$ only (with a correction $|\!\log\e|^{1/2}$ in even dimensions): the classical $\Ld^2$-based corrector theory is not accurate at the order $\e^{d/2}$ of fluctuations~\cite{GuM,DGO1}.
In contrast, for the averaged solution, Conjecture~\ref{conj:main} together with Proposition~\ref{prop:main-gen} implies an approximation result for $\expec{u_{\e,f}}$ up to order $O(\e^{2d-\eta})$ for all $\eta>0$. We briefly discuss the consequences of such a result in the context of fluctuations.

\subsection{Classical corrector theory}

We focus for simplicity on the model framework of a Gaussian coefficient field $\Aa$ with integrable covariance.
More precisely, for some $k\ge1$, let $a$ be an $\R^k$-valued Gaussian random field, constructed on some probability space $(\Omega,\Pm)$, which is stationary and centered, hence characterized by its covariance function
\begin{equation*}
c(x):=\expec{a(x)\otimes a(0)},\qquad c:\R^d\to\R^{k\times k}.
\end{equation*}
We assume that  the covariance function is integrable at infinity
$\int_{\R^d} |c(x)|\,dx \,<\,\infty$.
Given a map $h\in C^1_b(\R^k)^{d\times d}$, we define $\Aa:\R^d\to\R^{d\times d}$ by $\Aa(x)=h(a(x))$, and assume that it satisfies the boundedness and ellipticity properties~\eqref{eq:unif-ell}  almost surely. We (abusively) call such a coefficient field $\Aa$ \emph{Gaussian with integrable covariance}.

In this setting, we consider the corrector equations \eqref{e.cor-1}--\eqref{e.cor-4}, where the average $\int_Q$  on the unit cell $Q$ is replaced by the expectation $\expec{\cdot}$.
Based on~\cite{GNO-reg,GNO-quant,BFFO-17} (or alternatively~\cite{AKM-book,GO4} if $\Aa$ rather satisfies a finite range of
dependence assumption), we obtain the following optimal control of correctors
(cf.\@ also~\cite[Proposition~C.4]{BG} for a similar statement).

\begin{lem}\label{lem:gaus}
Let $d\ge1$, let $\Aa$ be Gaussian with integrable covariance, and set $\ell:=\lceil\frac {d}2\rceil$.
For all $0\le n <\ell$, there exist unique stationary solutions $\varphi^n,\sigma^n \in \Ld^2(\Omega,H^1_\loc(\R^d))$ of \eqref{e.cor-1}--\eqref{e.cor-4} with $\expec{(\varphi^n,\sigma^n)}=0$, whereas for $n=\ell$ there exist unique (non-stationary) solutions $\varphi^\ell,\sigma^\ell \in \Ld^2(\Omega,H^1_\loc(\R^d))$ such that $\nabla \varphi^\ell,\nabla \sigma^\ell$ are stationary and $(\varphi^\ell(0),\sigma^\ell(0))=0$ almost surely. In particular, $\bar \Aa^n$ is well-defined for all $1\le n\le\ell$.
In addition, for all $x\in \R^d$,
\begin{gather*}
\sum_{n=1}^{\ell-1} \|(\varphi^n,\sigma^n)\|_{\Ld^2(\Omega)}+\sum_{n=1}^{\ell}\|\nabla(\varphi^n,\sigma^n)\|_{\Ld^2(\Omega)}+\sum_{n=1}^\ell|\bar \Aa^n|\,\lesssim\, 1,\\
\|(\varphi^\ell,\sigma^\ell)(x)\|_{\Ld^2(\Omega)}\,\lesssim\,
\begin{cases}
\log^{\frac12}(2+|x|),&\text{if }d\text{ is even},\\
1+|x|^\frac12, &\text{if }d\text{ is odd.}
\end{cases}
\qedhere
\end{gather*}
\end{lem}

Mimicking the proof of Proposition~\ref{prop:cor-per}, we are then led to the following (cf.~\cite{Gu-17,DO1}).
Note that this corrector theory is not accurate at the order $\e^{d/2}$ of fluctuations.

\begin{prop}[Classical corrector theory --- random setting]\label{prop:cor-rand}
Let $d\ge1$, let $\Aa$ be Gaussian with integrable covariance, and set $\ell:=\lceil\frac d2\rceil$.
Given $f\in C^\infty_c(\R^d)^d$, let the $\ell$th-order homogenized solution $\bar u_{\e,f}^\ell$ for~\eqref{eq:ueps} be defined as in the statement of Proposition~\ref{prop:main-gen}(II) with homogenized coefficients defined in Lemma~\ref{lem:gaus}.
Then,
\begin{align*}
\Big\|\nabla \Big(u_{\e,f}-\sum_{k=0}^\ell\e^{k}\varphi_{i_1\ldots i_k}^k(\tfrac\cdot\e)\nabla_{i_1\ldots i_k}^k\bar u_{\e,f}^\ell\Big)\Big\|_{\Ld^2(\R^d\times Q)}\,\le\,\e^{\frac d2}\mu_d(\e) C\|\langle\nabla\rangle^{2\ell-1}f\|_{\Ld^2(\R^d)},
\end{align*}
where the constant $C$ depends only on $d,\lambda$ and where
\[\mu_d(\e):=\begin{cases}
|\!\log \e|^\frac12,&\text{if $d$ is even},\\
1,&\text{if $d$ is odd}.
\end{cases}\]
In particular,
\[\big\|\nabla(\expec{u_{\e,f}}-\bar u_{\e,f}^\ell)\big\|_{\Ld^2(\R^d)}\,\le\,\e^{\frac d2} \mu_d(\e)C\|\langle\nabla\rangle^{2\ell-1}f\|_{\Ld^2(\R^d)}.\qedhere\]
\end{prop}

\subsection{Consequences of Conjecture~\ref{conj:main}}\label{sec:fluc}

We now investigate the implications of Conjecture~\ref{conj:main}.
In view of Proposition~\ref{prop:main-gen}, it would lead to an effective approximation result for the averaged solution up to the accuracy $O(\e^{2d-\eta})$, which substantially improves on the above result obtained from the classical corrector theory.

\begin{cor}[of  Conjecture~\ref{conj:main}]\label{coro}
Let $d\ge1$ and let $\Aa$ be Gaussian with integrable covariance.
If Conjecture~\ref{conj:main} holds true, then the (symmetrized) higher-order homogenized coefficients $\bar\Aa^n$ are well-defined for all $1\le n\le2d$ via formula~\eqref{eq:link-B-a}.
For $1\le n\le\lceil\frac d2\rceil$, these coefficients coincide with the ones defined in \eqref{e.cor-2} via averages of correctors.
In addition, given $f\in C^\infty_c(\R^d)^d$, letting the $(2d)$th-order homogenized solution $\bar u_{\e,f}^{2d}$ for~\eqref{eq:ueps} be defined as in the statement of Proposition~\ref{prop:main-gen}(II), we have for all $\eta>0$,
\[\big\|\nabla(\expec{u_{\e,f}}-\bar u_{\e,f}^{2d})\big\|_{\Ld^2(\R^d)}\,\le \, \e^{2d-\eta}C_\eta\|\langle\nabla\rangle^{4d-1}f\|_{\Ld^2(\R^d)},\]
where the constant $C_\eta$ depends only on $d,\lambda,\eta$.
\end{cor}

In stochastic homogenization, there is a particular interest in the fluctuations of macroscopic observables of the type $U_\e(f,g):=\e^{-d/2}\int_{\R^d}g\cdot\nabla u_{\e,f}$ with $f,g\in\Ld^2(\R^d)^d$.
Such observables are asymptotically Gaussian and their limiting variance has been completely characterized in~\cite{GuM,DGO1,DO1}.
This should be complemented with a description of the expectation $\expec{U_\e(f,g)}$. While Proposition~\ref{prop:cor-rand} is not precise enough to describe $\nabla\expec{u_{\e,f}}$ in the fluctuation scaling, Corollary~\ref{coro} is, and yields
\[U_\e(f,g)=\big(U_\e(f,g)-\expec{U_\e(f,g)}\big)\,+\,\e^{-\frac d2}\int_{\R^d}g\cdot\nabla\bar u_{\e,f}^{2d}\,+\,O_{f,g,\eta}(\e^{\frac{3d}2-\eta}),\]
where the law of the first right-hand side term is close to a centered Gaussian with fully characterized variance, cf.~\cite{GuM,DGO1,DO1}.
The only difficulty left in this picture is the practical computation of the higher-order homogenized coefficients (although $B$ is well-defined, it is hardly computable in practice).
We believe that the following ``approximation by periodization'' should come out 
of the proof of Conjecture~\ref{conj:main}:
Let $d\ge1$, let $\Aa$ be a stationary ergodic coefficient field, and for all $L>0$ denote by $\Aa_L$ the random field
obtained by periodizing the restriction of $\Aa$ on $[-\frac L 2, \frac L2)^d$.
For all $n\ge1$, denote by $\bar \Aa_L^n$ the (random) $n$th-order homogenized coefficient associated with this periodic medium.
We conjecture that for all $1\le n\le 2d$,
\[\lim_{L\uparrow \infty} \expec{\bar\Aa_L^n} \,=\, \bar\Aa^n,\]
property that has recently been proved in~\cite[Remark~2.4]{DO1} in the limited range $n<d$ by a duality argument.
It would also be of interest to quantify this convergence (as done in~\cite[Theorem~2]{GNO1} for~$\bar\Aa^1$).

\section*{Acknowledgments}
We warmly thank Tom Spencer for attracting our attention to this problem.
The work of MD is supported by F.R.S.-FNRS and by CNRS-Momentum.
Financial support of AG is acknowledged from the European Research Council under the European Community's Seventh Framework Programme (FP7/2014-2019 Grant Agreement QUANTHOM 335410).

\appendix
\section{Proof of Lemma~\ref{lem:main}}\label{app}
The stationarity of the coefficient field $\Aa$ entails that the averaged solution operator $H:=\expec{\Hc}:\Ld^2(\R^d)^d\to\Ld^2(\R^d)^d$ commutes with translations on $\R^d$. Hence, $H$ is a convolution operator on $\Ld^2(\R^d)^d$. Noting that $Hf$ is gradient-like for all $f\in\Ld^2(\R^d)^d$ and that $H$ vanishes on solenoidal vector fields, we deduce that the Fourier symbol $\hat H$ takes the form
$\hat H(\xi)=\hat G(\xi)\xi\otimes\xi$,
for some measurable function $\hat G:\R^d\to\R$.
Now noting that the boundedness and the ellipticity~\eqref{eq:unif-ell} of $\Aa$ imply $H_0\le \Hc\le\frac1\lambda H_0$, where $H_0=\nabla\triangle^{-1}\nabla\cdot$ denotes the usual Helmholtz projection, we deduce $|\xi|^{-2}\le\hat G(\xi)\le\frac1\lambda|\xi|^{-2}$ pointwise.
Considering the inverse symbol $\hat B(\xi):=\hat G(\xi)^{-1}$,
the conclusion follows. 
\qed

\bibliographystyle{plain}
\bibliography{biblio}

\end{document}